\documentclass{article}
\usepackage{amssymb}

\usepackage{graphicx}
\usepackage{amsmath}
\usepackage{color}
\usepackage{float}
\usepackage{array}

\newtheorem{theorem}{Theorem}[section]

\newtheorem{corollary}[theorem]{Corollary}

\newtheorem{definition}[theorem]{Definition}

\newtheorem{lemma}[theorem]{Lemme}

\newtheorem{proposition}[theorem]{Proposition}
\newtheorem{remark}[theorem]{Remark}

\newenvironment{proof}[1][Proof]{\textbf{#1.} }{\ \rule{0.5em}{0.5em}}
\def\O{\Omega}

\def\dpa{\partial}

\def\D{\Delta}

\newcommand{\R}{\mathbb{R}}

\newcommand{\cqfd}

\title{On Pompeiu's-Schiffer's Conjectures  from  Shape Optimization}
\author{ Diaraf SECK$^{1,\,}$ \footnote{diaraf.seck@ucad.edu.sn}\\\\
$^{1}$ Laboratoire de Math\'ematiques de la D\'ecision et \\
d'Analyse Num\'erique (LMDAN), BP 16889, Dakar Fann, S\'enegal\\
Ecole Doctorale de Math\'ematiques et Informatique U.C.A.D. Dakar,  S\'en\'egal.}

\begin{document}
\maketitle
\begin{abstract}
Our aim is to  do a come back on Schiffer's and Pompeiu's conjectures with shape optimization tools, maximum principles and Serrin's symmetry method. We propose  a way to  get affirmative answers in some cases. We propose also sufficient conditions thanks to Riemannian  approach of infinite dimension that could be useful for numerical simulations of the shape of domains related  to these conjectures. 
\end{abstract}
{\bf Keywords:}  shape optimization, maximum principle, anti maximum principle, comparison principle, moving planes method, eigenvalue, symmetry, infinite Riemannian manifolds.\\
\section{Introduction} 
Dimitrie Pompeiu (1873-1954) was born in Romania.  If my information is accurate, he got his Ph.D in 1905 at the Sorbonne, in Paris under the supervision of Henri Ponicar\'e. 
He is known mainly for the   Pompeiu- Hausdorff metric,  Pompeiu problem and for the Cauchy-Pompeiu formula in complex analysis.\\
We formulate the Pompeiu problem as it is understood today.\\
Let $f\in L_{loc}^1 (\R^N)\cap \mathcal S'$ where  $\mathcal S'$ is the Schwartz class of distributions and 
\begin{eqnarray}\label{Pomp}
\int_{\sigma (D)} f(x) dx= 0\; \;  \forall \sigma \in G,
\end{eqnarray}
where $G$ is the group of all rigid motions of $\R^N$, consisting of all translations and rotations, and $D\subset  \R^N$ is a bounded domain, the closure $\overline{D}$ of which is diffeomorphic to a closed ball.\\
Does ($\ref{Pomp}$) imply that $f = 0?$
This question was raised in \cite{Po}: \\
If yes, then we say that $D$ has $P-$ property (Pompeiu' s property), and write $D \in P.$\\
Otherwise, we say that D fails to have $ P-$property, and write $ D \in \overline{P}. $ \\
Pompeiu claimed in 1929 that every plane bounded domain has $P-$property. This claim turned out to have a gap. A counterexample was given 15 years later by Chakalov, and for more information see  \cite{Ch}. \\
A domain $\Omega\subset \mathbb{R}^2$ is said to have the Pompeiu property if $f \equiv 0$ is the only continuous function in $\mathbb{R}^2$ such that the integral of $f$ over $\sigma (\Omega)$  , for every rigid motion $\sigma$ of $\mathbb{R}^2$ , vanishes. It has been conjectured that the disc is the only bounded simply connected domain, modulo sets of Lebesgue measure zero, in which the Pompeiu property fails. \\
By a theorem of L. Brown, B.M. Schreiber and B.A. Taylor \cite{BST},  a bounded domain $\Omega$ has the Pompeiu property if and only if $\hat{ \mu}$ , the Fourier- Laplace transform of the area measure $\mu$ of $\Omega$, does not vanish identically on
\begin{eqnarray*}
M_{\alpha}= \{ (\zeta_1, \zeta_2)\in \mathbb{C}^2: \zeta_1^2 + \zeta_2^2= \alpha\}
\end{eqnarray*} 
for any $\alpha \in \mathbb{C}^*= \mathbb{C}\backslash \{0\}.$\\
Also, in $1976, $ S.A. Williams \cite{Wi1}, showed that if $\Omega$ is a bounded simply connected Lipschitz domain then it has the Pompeiu property if and only if there is no solution to the following overdetermined Cauchy problem
\begin{eqnarray*}
\begin{cases}\Delta u + \alpha u= 1\; \;  in \; \;  \Omega\\
\frac{\partial u}{\partial \nu}_{\big\arrowvert_{\partial \Omega}}= 0\\
 u_{\big \arrowvert_{\partial \Omega}}=  0\\
\end{cases}
\end{eqnarray*}
where the Laplace operator $\Omega$ is in two dimensions, for any $\alpha \in \mathbb{C}^*$  (in fact, it suffices to consider $\alpha > 0,$ cf. \cite{B}). It is  to underline that, in this formulation  of the problem essentially goes back to the original book on “The theory of sound” by Lord Rayleigh. It later became known as the Schiffer problem and in this context the conjecture mentioned above is known as Schiffer's conjecture \cite{GS1} Another result of Williams is that any Lipschitz domain in which the Pompeiu property fails must have a nonsingular analytic boundary (see \cite{Wi2}). Consequently, as regards the Pompeiu problem, assuming that the boundary of a domain is nonsingular and analytic is not excessive. It is mentioned also in \cite{Ya}, Problem 80, p. 688 as an open  question. And until now, we are not yet aware that it is solved.

In $1993$, P. Ebenfelt  \cite{Eb}, obtained some results which support that conjecture. He showed that the disc is the only quadrature domain in which the Pompeiu property fails. Also he proved a result claiming nonexistence, under certain conditions, of solutions to a family of overdetermined Cauchy problems. This result is used to obtain the Pompeiu property for a wide variety of domains, including “kth roots” of ellipses and domains which are mapped conformally onto the unit disc by a rational function other than a M$\ddot{o}$bius transformation.\\
Now, we recall the current formulation of the P -problem is the following:\\
Let us make the following standing assumptions:
Assumptions $ A:$\\
$A1)$  $\Omega$ is a bounded domain, the closure of which is diffeomorphic to a closed ball, the boundary $S:= \partial \Omega$ of $\Omega$ is a closed connected $\mathcal  C^1-$ smooth surface,\\
 $A2)$ $\Omega$ fails to have $P-$ property.\\
\textbf{Conjecture 1:} If Assumptions A hold, then $\Omega$ is a ball.\\
\textbf{Conjecture 2:} If problem $$\Delta u + k^2 u= 1\;  in\;   \Omega, \frac{\partial u}{\partial \nu}_{\big\arrowvert_{\partial \Omega}}= 0 , u_{\big \arrowvert_{\partial \Omega}}= 0, k^2 = const > 0 $$
has a solution, then $\Omega$ is a ball.\\
This  is an open symmetry problem of long standing for partial differential equations.\\
\textbf{Conjecture 3:} If Assumption $ A1)$ holds and the Fourier transform $\hat{\chi}_{\Omega}$ of the characteristic function $\chi_\Omega$ of the domain $\Omega$ has a spherical surface of zeros, then $\Omega$ is a ball.\\
\textbf{Conjecture 4:}  Let $const$ be a given constant. If the problem 
\begin{eqnarray*}
\begin{cases}\Delta u + k^2 u= 0\; \;  in \; \;  \Omega\\
\frac{\partial u}{\partial \nu}_{\big\arrowvert_{\partial \Omega}}= 0\\
 u_{\big \arrowvert_{\partial \Omega}}= const \neq 0\\
 k^2 = \lambda(\Omega)> 0 
\end{cases}
\end{eqnarray*}
has a solution, then $\Omega$ is a ball.\\
All the above four conjectures are equivalent. And these  symmetry problems are known as the Schiffer's conjectures.

But there is another symmetry problem for partial differential equations included in the  M. Schiffer's conjectures:\\
\textbf{Conjecture 5:} 
Let $c\neq 0$ be a  given  real constant,
if 
\begin{eqnarray*}
\begin{cases}\Delta u + k^2 u= 0 \; \; in\; \;   \Omega\\
 \frac{\partial u}{\partial \nu}_{ \big\arrowvert_{\partial \Omega}}= c   \\
u_{\big \arrowvert_{\partial \Omega}}= 0 \\
k^2 = \lambda(\Omega) > 0 
\end{cases}
\end{eqnarray*}
  has a solution, then $\Omega$ is a ball.
\begin{remark}
It is important to underline that Conjecture 5 is not equivalent to Conjecture 4.
\end{remark}
There is  a list of interesting results about the Pompeiu problem. We would like to quote some of references containing these contributions (\cite{GS1}, \cite{Za}, \cite{B1}, \cite{B}, \cite{Dal1}, \cite{De}, \cite{FMW}, \cite{ChHe}, \cite{Li}, \cite{KM}, \cite{NSY}, \cite{BY}, \cite{Av}, \cite{Mo}).\\
Note that in \cite{Ke}, the author shows connections that exist between  the Pompeiu problem, stationary solutions to the Euler equations, and the convergence of solutions to the Navier-Stokes equations to that of the Euler equations in the limit as viscosity vanishes. \\
And in \cite{ChHe}, the authors use shape optimization tools to obtain   partial positive  answer. They  show the connection between these problems and the critical points of the functional eigenvalue with a volume constraint. They  use this fact, together with the continuous Steiner symmetrization, to give another proof of Serrin's result for the first Dirichlet eigenvalue. In two dimensions and for a general simple eigenvalue, they obtain different integral identities and a new overdetermined boundary value problem.

Our aim is to use shape optimization tools combined with maximum principles and a point of view of shape optimization in  an   infinite Riemanian framework to give some positive answers.
We give a theoretical  trial to open the numerical way in order to see if it is possible to give  framework which could permit to do simulations on these conjectures.\\
The paper is organized as follows: in the next section, we shall give a brief overview of basic but fundamental results on maximum principles. We will give a sufficient condition to  apply maximum principle theory for eigenvalue Laplace  problems. The section 3, is devoted to a  symmetry result of domains. It relies  on the Alexandrov moving planes  and the seminal paper due to J Serrin, \cite{Ser}. In section 4, we present the first main result on the  existence result of the Schiffer's problem (conjecture 5). In this section we use shape optimization theory and  the two previous sections.
In the last section, we propose  a study on necessary and sufficient conditions to  get positive answer for the Pompeiu problem. We will combine classical shape calculus and its Riemannian point of view in infinite dimension.

\section{Some basic tools on maximum principles} 
In this section, we intend  to give an overview of basic but important results of maximum principles. We are note going to give their  proofs. For more details the reader is invited to see for instance \cite{PW}, \cite{GT}, \cite{PS}, \cite{Ev}.\\ 
Let  $\Omega$ be  an open set of $\mathbb{R}^N$, consider the following elliptic operator
$L$ defined:
$$Lu  =  
  \displaystyle\sum_{i,j=1}^N{a_{ij}(x)\frac{\partial^2u}{\partial x_i\partial x_j}}+\displaystyle \sum_{i=1}^N{b_i(x)
\frac{\partial u}{\partial x_i}}+c(x)u \; \; x \in \Omega $$
where $a_{ij}=a_{ji}\,\,\in \,\,\mathcal{C}(\overline{\Omega}),\,\,\,\,
 c, b_i  \in \mathcal{C}(\overline{\Omega})\;$ if  $\Omega$ is bounded or $\mathcal{C}(\overline{\Omega})\cap L^{\infty} (\Omega)$ if $\Omega $ is unbounded.

and  there exists \,\,\, $c_0,C_0,\,\,\,\, 0\,\,<\,\,c_0\,\,<\,\,C_0 $\\
 such that  for all  \,\,$x\,\,\,\in \,\,\Omega\,$ and \, $\,\, \xi  \,\, \in\,\,\R^N,$ we have
$$c_0 |\xi |^{2}\,\,\leq  \,\, a_{ij}(x){\xi}_{i} {\xi}_{j} \,\,\leq \,\, C_0 |\xi|^{2}.$$
\begin{remark}
\begin{itemize}
\item The sign of  $c(x)$ plays a main role in the Maximum principles in many cases.
\item If $ \;  \, b_j=  \displaystyle\sum_{i=1}^N\frac{\partial a_{ij}(x)}{\partial x_i} \,\,\,\,$   then $L$ is a divergence form.
\end{itemize}
\end{remark}
\subsection{Weak maximum principle}
$$L= M+ c (x) I$$
where 
$$M  =  
  \displaystyle\sum_{i,j=1}^N{a_{ij}(x)\frac{\partial^2}{\partial x_i\partial x_j}}+\displaystyle \sum_{i=1}^N{b_i(x)
\frac{\partial }{\partial x_i}}$$
\begin{theorem}
Let  $\Omega$ be a bounded open  set  let us consider $M.$\\
Let $u \in \mathcal{C}^2 (\Omega)\cap\mathcal{C} (\overline{\Omega}), $ such that $Mu \geq 0$ in $\Omega.$
Then  $$\max_{\overline{\Omega}}u= \max_{\partial \Omega} u$$
\end{theorem}
\begin{theorem}
Let us consider $\Omega,$  a bounded open set,  $L$ as above with $c(x)\leq 0.$
One supposes that $u \in \mathcal{C}^2 (\Omega)\cap \mathcal{C} (\overline{\Omega}), $ with $Lu\geq 0$ in $\Omega$
\\
Then $$\max_{\overline{\Omega}} u= \max_{\partial \Omega} u^{+}$$

\end{theorem}
\begin{corollary}
Let us consider $\Omega,$  a bounded open set,  $L$ as above with $c(x)\leq 0.$
One supposes that $u \in \mathcal{C}^2 (\Omega)\cap \mathcal{C}(\overline{\Omega}), $ with $Lu\leq  0$ and $u \geq 0$ on $\partial \Omega$ then
$$u\geq 0 \; \; \mbox{in}\; \;  \Omega.$$
\end{corollary}
\begin{definition}
One supposes $c : \Omega \rightarrow  \mathbb{R},$ ($c$ is not necessary positive). The Maximum Principle (MP) is satisfied for $L$ in $\Omega$ if $\forall w \in \mathcal C^2 (\Omega)\cap \mathcal C (\overline{\Omega}) $ such that
\begin{eqnarray*}
\begin{cases}\-L w \leq 0\; \;  in \; \;  \Omega\\
w\geq  0 \, \, on \, \; \partial \Omega\\
\end{cases}
\end{eqnarray*}
then  $w \geq 0$ in $\Omega$
\end{definition}
\begin{remark}
If $c(x)\leq 0$ in $\Omega,$ then $L$ satisfies the maximum principle.
\end{remark}
\begin{theorem}
If  there is $g \in \mathcal C^2 (\Omega) \cap \mathcal C (\overline{\Omega})$ such that $g>0$ on $\overline{\Omega}$ and $ Lg \leq 0,$  then $L$ satisfies the  maximum principle in $\Omega.$
\end{theorem}
On can improve the above theorem as follows:
\begin{theorem}
If  there exists $g \in \mathcal C^2 (\Omega) \cap \mathcal C (\overline{\Omega})$ such that $g>0$ on $\Omega, Lg \leq 0$ in $\Omega,$ and $  g_{ \big\arrowvert_{\partial \Omega}} \not\equiv 0$ or $ Lg  \,\;  \not\equiv 0,$ then $L$ satisfies the  maximum principle in $\Omega.$
\end{theorem} 
There is  the maximum principle for the thin domains.
\begin{theorem}
$\exists \delta:= \delta (c_0, b, n)>0$ such that  if $L$ is an operator such that $a_{ij}(x)\geq c_0 Id, $ and there is $b>0$ such that both $\|b_i\|_{L^{\infty}}, \|c\|_{L^{\infty}}\leq b$ and if $\Omega$ is contained in a region $$\mathcal R= \{x \in \Omega ; a < \xi . x < a+ \delta, \xi \in \mathbb{S}^{n-1}\}; \; \, \mbox{where}\, \, \mathbb{S}^{n-1} \, \, \mbox{is the unit sphere of}\, \ \mathbb{R}^n, $$ then $L $ satisfies the maximum principle in $\Omega.$
\end{theorem}
There is also the maximum principle for the small domains
\begin{theorem}
Let $R>0$ be given and big enough, $ \exists \delta:=  \delta (c_0, b, n, R)>0$ such that if $\Omega \subset B_R (0)$ and $meas (\Omega)< \delta$ then  the maximum principle is satisfied by $L$ in $\Omega.$
\end{theorem}
Considering the following eigenvalue problem
\begin{eqnarray*}
\begin{cases}\-L \phi_1= \lambda_1 \phi_1\; \;  in \; \;  \Omega\\
\phi_1= 0 \, \, on \, \; \partial \Omega\\
\end{cases}
\end{eqnarray*}
 At first, we begin by the following simple examples for the computation of the fundamental eigenvalue.
\begin{enumerate}
\item If $Lu= u" + \pi^2u, \Omega= (0, 1), u(0)= u(1)= 0,$  then $\lambda_1 (-L, \Omega)= 0.$\\
If $Lu= u" + c (x)u, u(0)= u(1)= 0$ with  $c(x)\leq \pi^2$ and $c(x)\neq \pi^2$ then $\lambda (-L, (0, 1))> 0.$
\item   If $Lu= u" + \pi^2 u, u(0)= u(a)= 0, \, \Omega= (0, a), a \neq 1,$ then  $ \lambda_1 (-L, \Omega)= \pi^2 (\frac{1}{a^2}- 1)$
\item If  $Lu= u" + k u,, u(0)= u(a)= 0, \, \Omega= (0, a),$ then  $ \lambda_1 (-L, \Omega)= \frac{\pi^2}{a^2}- k$
\end{enumerate}
\begin{theorem}
$L$ verifies the maximum principle iif $\lambda_1 > 0$
\end{theorem}
Let us consider $L= \Delta + k^2 I, \Omega, $  a bounded regular domain in $\mathbb{R}^N.$
and  the following eigenvalue problem
\begin{eqnarray}
\begin{cases}-\Delta v_1 - k^2 v_1= \lambda v_1 \; \; \mbox{in}\; \;   \Omega\\
v_1= 0 \; \; \mbox{on} \; \; \partial \Omega.
\end{cases}
\end{eqnarray}
Then, thanks to the  above theorem, the maximum principle  is satisfied iif  
\begin{eqnarray*}
\lambda_1= \inf \{\int_{\Omega}|\nabla v_1|^2- k^2, v_1 \in H_0^1 (\Omega)\backslash \{0\}, \int_{\Omega}v_1^2 dx= 1\}> 0.
\end{eqnarray*}
Otherwise,
\begin{eqnarray*}
\lambda_1= \int_{\Omega}|\nabla \phi_1|^2 dx - k^2> 0, \int_{\Omega} \phi_1^2 dx= 1,  \phi_1 \in H_0^1 (\Omega)\backslash \{0\}. 
\end{eqnarray*}
We consider  also   the following problem in the ball centered at origin and of radius $R, B (0, R):= B_R\subset \mathbb{R}^2$
\begin{eqnarray}
\begin{cases}-\Delta u = \alpha  u \; \; \mbox{in}\; \;   B_R\\
u= 0 \; \; \mbox{on} \; \; \partial B_R.
\end{cases}
\end{eqnarray}
The eigenvalues of the above problem are : $\alpha  = (\frac{j_{n, m}}{R})^2$ for $n \geq 0, m\geq 1$ where $j_{n,m}$ are the $m-{th}$  positive roots of the Bessel function of order $n,$ $J_n (r).$
 And $\alpha_1:=  (\frac{j_{n, 1}}{R})^2$ is the smallest one.\\
 The next step  we aim to discuss in this work is, if we set    $\alpha_1 =k^2+ \lambda_1, $ for which condition $\lambda_1> 0?$\\
 We can see that $\lambda_1 >0$ if and only if $\alpha_1> k^2.$ \\
 All our work (Sections $3$ and $4$) will rely on this above sufficient  and necessary condition which leads us to use maximum principles.
 \begin{remark}
 If  $k$ is such that, $|k|< \frac{j_{n, 1}}{R}$ then $\lambda_1> 0.$\\
 If $k$ is fixed at first, playing on the values of $R; 0< R< \frac{j_{n,1}}{|k|},$ we have  $\lambda_1>0.$
 \end{remark}

\subsection{Strong  maximum principle}
\begin{definition}Interior sphere condition (ISC):\\
An open set  $\Omega \subset \mathbb{R}^N$  (or its boundary )satisfies  (ISC) if:\\ 
$\forall p \in \partial \Omega, \exists B= B_{\rho}(a)$ (a ball centered in $a$ with radius $\rho> 0$) such that 
$$B \subset \Omega,\;  \mbox{and }\; p \in \partial B.$$ 
\end{definition}
\begin{lemma}
Let  $\Omega$ be a bounded open set, $p\in \partial \Omega$ and $\Omega$ satisfying (ISC).
Suppose that  $u \in \mathcal C^2 (\Omega)$ may be extended by continuity in $p$ with value equals $u(p)$ and such that $Mu \geq 0$ in $\Omega,$\\  $u(p)> u(x)\; \;  \forall x \in \Omega.$\\
Let $B= B_{\rho}(a) $ the interior sphere such that $p\in \partial B$ and $\xi$ an outward direction in $p$\\ ($<\xi, p-a>>0$), then 
$$\lim_{t \searrow 0^{+}}inf \frac{u(p)- u(p-t\xi)}{t}>0\; \;  (t> 0)$$
\end{lemma}
\begin{theorem} (Strong maximum principle)\\
Let us consider $\Omega$ be a connected  bounded open set,  $L$ as above with $c(x)\leq 0.$
One supposes that $u \in \mathcal C^2 (\Omega)$ and  $Lu \geq 0$ in $\Omega.$
\begin{itemize}
\item  Case: $c\equiv 0$: if $u$ reaches its maximum in $\Omega$ then  $u$ is a constant
\item Case: $c (x)\leq 0$ If $u$ reaches its maximum in $\Omega$ and  this maximum is non negative then $u$ is constant.
\end{itemize}
\end{theorem}
\begin{theorem} (Strong boundary  maximum principle)\\
Let us consider $\Omega$ be a connected  bounded open set,  $L$ as above with $c(x)\leq 0.$
One supposes that $u \in \mathcal C^2 (\Omega)$ and continuous at $p\in \partial \Omega$ and  $Lu \geq 0$ in $\Omega $ with $u(p)= \displaystyle \max_{\overline{\Omega}}u.$\\
In addition, one supposes that $\Omega $ satisfies (ISC) at $p$ and if $c\not\equiv 0$,  $u(p)\geq 0.$
Then, we have:
\begin{itemize}
\item or $u$ is constant
\item or  $\frac{\partial u}{\partial \xi}> 0$ where $\xi$ an outward direction in $p$
\end{itemize}
\end{theorem}
\subsection{Antimaximum principle}
In our study, maximum principle plays a key to reach our aim. 
Seminal works in \cite{CP} on an antimaximum principle for second order elliptic  shows that the powerful tool of maximum principle fails in front of simple but interesting questionings. The authors proved the antimaximum principle for a general class of linear boundary value problem of the form
\begin{eqnarray*}
\begin{cases}L u -\mu u = f\; \;  in \; \;  \Omega\\
 Bu= 0  \, \, on \; \; \partial \Omega

\end{cases}
\end{eqnarray*}
 where $\Omega$ is a bounded  domain in $\mathbb{R}^N$ with smooth boundary $\partial \Omega, \mu \in \mathbb{R}, f$  is a  function sufficiently smooth defined  on $\Omega.$ $L$  denotes a second order elliptic differential operator and $B$ a first order boundary operator, for more details see \cite{CP}.\\
Let us recall  the result of the following  simple but instructive example. One considers   $b$ a function sufficiently smooth defined $\partial \Omega$ with $b\geq 0.$
 \begin{eqnarray*}
\begin{cases}- \Delta u -\mu u = f\; \;  in \; \;  \Omega\\
\frac{\partial u}{\partial \nu} + bu= 0  \, \, on \; \; \partial \Omega
\end{cases}
\end{eqnarray*}
Let $\mu_0$ be the principal eigenvalue of the Laplace operator $-\Delta$ on $\Omega.$ It is well known that if $\mu < \mu_0,$ the strong maximum principle holds: if $f(x)\geq 0 (\not \equiv 0)$  in $\Omega,$ then $u(x)>0$ for any $x \in \Omega.$
For certain values of $\mu > \mu_0,$ the complete opposite of the maximum principle holds: given $f(x)\geq 0 (\not \equiv 0),$ there exists  $\delta >0$ such that if $\mu_0 < \mu < \mu_0 + \delta,$ then $u(x)< 0$ for any $x \in \Omega.$\\
There are  numerous situations where the maximum principle cannot be used and the moving plane techniques are not an adequate argument to bring responses on symmetry in partial differential equations problems. \\
In our work we shall focus on the situations where maximum principle and moving planes techniques hold.
\section{Symmetry}
The symmetry problems of domains had  a great interest at least  fifty years ago. And until our days, they continue to attract much interest. One can mention the following  references as  first famous  symmetry results for domains: \cite{Ser}, \cite{Wein}, \cite{Re}, \cite{GL},\cite{BH},\cite{Ra}  and \cite{BN}   and references therein. 

The theorem proved in the here belongs to the family of symmetry results introduced by J. Serrin. And it is proposed under the hypothesis of Maximum principle. This mean that we suppose thta the first eigenvalue $\lambda_1$  of the operator $-\Delta  - k^2 I $ is positive.
\begin{theorem}\label{thsym}
Let
\begin{enumerate}
\item $\Omega$  be an open and bounded set of  $\R^N$ contenant $K$ with  $\partial \Omega$ of class $\mathcal C^2;$
\item $K$  be symmetric with respect to the hyperplane  that we call  and expressed by  $T_0= \{x_N= 0\};$
\item One supposes that there is  a solution   $u\in  \mathcal C^2 (\bar{\Omega}\backslash K)$ of  the following overdetermined problem
\begin{eqnarray}
\displaystyle \left\{\begin{array}{c c c c c}
\Delta u + k^2 u &=&0&  \mbox{in} &\Omega \backslash K\\
u &=&  1&\mbox{on} &\partial K \\
u &=&0 &\mbox{on} & \partial \Omega\\
|\nabla u| &=& c_1& \mbox{on} &  \partial \Omega\\
(c_1 > 0)
\end{array}\right.
\end{eqnarray}
\item  In addition we suppose that  $K$ is convex in the  direction $x_N.$
\end{enumerate}
Then \\
$\Omega$ is  symmetric with respect to   the hyperplane  $T_0. $ Moreover $u$ is symmetric with respect to $T_0.$
\end{theorem}
We are going to  give the main  steps and keys for the proof of this result. \\
At first we begin by setting:
\begin{enumerate}
\item  $a = \displaystyle \inf_{x=(x_1,\cdots, x_N)} x_N;$  and  we can always suppose that $a$ is negative because of the symmetry hypothesis on $K$ and the fact that $T_0= \{x_N=0\}$
\item $T_{\lambda}$ the hyperplane characterized by  $(x_N=\lambda).$ We quote that  $T_{\lambda}$ is parallel to  $T_0$for any $\lambda$;
\item Let $(\mathcal P_0)$ be the following property:
$$\forall \lambda \in \R, a \leq \lambda \leq 0; \sigma_{\lambda}(K_{\lambda}^{+})\subset K_{\lambda}^{-}$$
where  $ K_{\lambda}^{+}$ is the  part of K  $K$ situated at the top side of   $T_{\lambda},$ $\sigma_{\lambda}$ is the orthogonal symmetry with respect to  $T_{\lambda}, $ $\sigma_{\lambda}(K_{\lambda}^{+})$  is the symmetric set of $ K_{\lambda}^{+}$ with respect to  $T_{\lambda}$  and   $K_{\lambda}^{-}= K\backslash K_{\lambda}^{+}.$
\end{enumerate}
Then we have the following proposition.
\begin{proposition}\label{prop}
$K$ is  convex in the direction of  $x_N$ if and only if  $(\mathcal P_0)$is satisfied.
\end{proposition}
\begin{proof}
To show that $\mathcal P_0$ is a neceessary condition, we are going to do a reasoning by absurd. \\
And then $$\exists \lambda_0 \in [a, 0] \; \; \mbox{such that} \; \; \sigma_{\lambda_0}(K_{\lambda_0}^+) \not\subset K_{\lambda_0}^- . $$
It is translated  by: $ \exists  y \in \sigma_{\lambda_0}(K_{\lambda_0}^+) :  y \notin K_{\lambda_0}^-  $ and therefore $y \notin K.$\\
For $x \in  \sigma_{\lambda_0}(K_{\lambda_0}^+),$ then  we have : $\exists x' \in  K_{\lambda_0}^+:  \; \;  \sigma_{\lambda_0}(x')=x. $\\
Let $x''$ be the symmetric point of $x'$ with respect to $T_0,$ then we claim that $x'\in K  \Rightarrow x'' \in K.$\\
Since the hyperplanes $T_0$ and $T_{\lambda_0}$ are parallel, $x, x' \; \mbox{and}\; x''$ belong to a same right lenght which is parallel to the axis $(Ox_N)$ where $O$ is the origin of the considered  orthonormal reference.\\
However, we get $[x', x''] \not\subset K$ and this implies that $K$ is not convex in the direction of $x_N:$ that is  a contradiction with the hypothesis.\\
Let us show now that que $\mathcal P_0$ implies that $K$ is convex in the direction of $x_N.$ For this let's suppose that K is not convex in the direction of $x_N.$\\
Then  there are $x, y$ lying in the same  straight line such that $x\in K, y\in K$ but $[x,y]\not\subset K.$ This means that $\exists x_0 \in [x, y]: x_0 \notin K.$
If we consider the $Nth$ components of $x$ and $x_{0}$, let's take  $\lambda_0= \frac{x_N+ x_{0N}}{2}.$ Then because of the fact that $d_2 (x, x')= d_2 (x, x_0)$, $d_2$ where being  the Euclidian distance and $x'=(x_1, x_2, \cdots, x_{N-1}, \lambda_0.)$,  we have $\sigma_{\lambda_0} (x)= x_0.$ \\
Depending on the direction of the axis that one  chooses, we have
\begin{itemize}
\item $\lambda_0 < 0$ if  $x$ and $x_0$ are situated in the same side of  $T_0;$
\item $ \lambda_0= 0$ if $x$ and $x_0$ are located on either side of $T_0.$
\end{itemize}
It is  easy to remak that $\lambda_0 \geq a$ and then $\lambda_0 \in [a, 0].$
From all the above justifications, one deduces that\\
$x \notin  K_{\lambda_0}^-$ and $x \in K_{\lambda_0}^+.$\\
Hence we conclude that  $\sigma (x) \notin K_{\lambda_0}^-$  and therefore $\sigma(K_{\lambda_0}^+) \not\subset K_{\lambda_0}^-.$
\end{proof}

Before proving the  theorem, we need the following lemma. And its proof can be found in \cite{Ser}.
\begin{lemma}\label{lemse}
Let $\Omega$ be a $\mathcal C^2$ regular domain of $\mathbb{R}^N, $ $T$ the hyperplane containing the  normal vector  $\vec{n}$ to $\partial \Omega$  at some point $x_0\in \partial \Omega.$ Let $\mathcal D $ un subset of $\Omega$  being in only one side of $T.$ One supposes that there is $w \in \mathcal C^2 (\overline{\mathcal D}),$ verifying:
\begin{eqnarray}
\displaystyle \left\{\begin{array}{c c c c c}
\Delta w  &\leq &0&  \mbox{in} &\mathcal D\\
w &\geq &  0&\mbox{in} &\mathcal D \\
w(x_0) &=&0 
\end{array}\right.
\end{eqnarray}
Then one of the following two conditions is satisfied if $w \not \equiv 0$
\begin{eqnarray*}
\begin{cases}
\mbox{(i)}\, \,  \frac{\partial w}{\partial \nu } (x_0)> 0\\
\mbox{(ii)}\, \,  \frac{\partial^2w}{\partial \nu^2} (x_0)> 0
\end{cases}
\end{eqnarray*}

\end{lemma}

\begin{proof}\textbf{ of the Theorem \ref{thsym}}\\
There are two exclusive  possibilities
\begin{itemize}
\item   $\sigma_{\lambda}(K_{\lambda}^{+})$ becomes internally tangent to  $\partial \Omega$ at some point $y_0$ with  $y_0 \notin T_{\lambda}$
\item $T_{\lambda}$ reaches a position where it is orthogonal to $\partial \Omega$ at some point  $x_0$, which
necessarily belongs to the closure of the strip between $T_0$ and $ T. $
\end{itemize}
Having at hand this proposition,it suffices to reproduce the  scheme of Serrin's proof (Alexandov's moving planes and maximum principle)  and that's all.\\
Our aim is to show that $\Omega$ is symmetric with respect to $T_0.$\\
Displacing the hyperplan $T_{\lambda}: x_N= \lambda$ in the sense of the axis $x_N$ or in the opposite sense and  parallel to the hyperplane $T_0,$ on may be in the two following cases:
\begin{enumerate}
\item $\sigma (K_{\lambda}^{+})$ is internally tangent to $\partial \Omega$ at a point $y_0, y_0 \notin T_{\lambda};$
\item $T_{\lambda}$ is orthogonal to $\partial \Omega$ at  a point $x_0.$
\end{enumerate}
Otherwise: $\exists \lambda_0 \in \mathbb{R}: \sigma_{\lambda_{0}} (\Omega_{\lambda_0}^+) \subset \Omega$ and
\begin{itemize}
\item  the above first case 1, occurs or
\item the  above second one  (case 2), occurs.
\end{itemize}
where $\Omega_{\lambda_0}^+$ is  the subset of $\Omega$ which is completely above  $T_{\lambda_0}.$
If we want to change the orientation of $x_N,$ we can assume that $\lambda_0\leq 0.$\\
Let us set $\Sigma_{\lambda_0}:=   \sigma_{\lambda_{0}} (\Omega_{\lambda_0}^+), $ let $v$ be the function defined on $\Sigma_{\lambda_0} \backslash K$ defined by:
\begin{eqnarray*}
\forall x  \in \Sigma_{\lambda_0} \backslash K,  v (x)= u (x') \, \, \mbox{where}\, \, x'= \sigma_{\lambda_0}(x)
\end{eqnarray*}
so that we have 
\begin{eqnarray}
\displaystyle \left\{\begin{array}{c c c c c}
\Delta v + k^2 v &=&0&  \mbox{in} &\Sigma_{\lambda_0} \backslash K\\
v(x) &=&  u(x')&\mbox{on} &\partial K\cap \Gamma_{\lambda_0} \\
v(x) &=&u(x) &\mbox{on} &\Gamma_{\lambda_0}\cap T_{\lambda_0}\\
v&=& 0&\mbox{on}& (\Gamma_{\lambda_0}\backslash K)\cap T_{\lambda_0}^c\\
|\nabla u| &=& c& \mbox{on} & (\Gamma_{\lambda_0}\backslash K)\cap T_{\lambda_0}^c
\end{array}\right.
\end{eqnarray}
where $ \Gamma_{\lambda_0}= \partial (\Sigma_{\lambda_0}\backslash K).$\\
Let us consider now the function $u-v$ which is also defined on  $\Sigma_{\lambda_0}\backslash K$ and satisfied
\begin{eqnarray}
\displaystyle \left\{\begin{array}{c c c c c}
\Delta (u-v) + k^2 u-v &=&0&  \mbox{in} &\Sigma_{\lambda_0} \backslash K\\
(u-v)(x) &=&  1 - u(x') &\mbox{on} &\partial K\cap \Gamma_{\lambda_0} \\
(u-v)(x) &=&0 &\mbox{on} &\Gamma_{\lambda_0}\cap T_{\lambda_0}\\
u-v&=& u&\mbox{on}& (\Gamma_{\lambda_0}\backslash K)\cap T_{\lambda_0}^c
\end{array}\right.
\end{eqnarray}
There are two possible situations, $\lambda_0 <0$ or $\lambda_0= 0$.
\begin{enumerate}
\item If  $\lambda_0 < 0, $ then thanks to maximum principle, we have: $u-v >0 \, \mbox{in}\,  \Sigma_{\lambda_0}\backslash K.$
\item If $\lambda_0=0,$ then: 
\begin{itemize}
\item $\Omega$ is symmetric with respect to $T_0,$ which completes the proof;
\item or $\Omega$ is not symmetric. And we deduce from the maximum principle that

  $u-v >0 \, \mbox{in}\,  \Sigma_{\lambda_0}\backslash K.$ But we are going to show that this latter cannot be realized.
  \end{itemize}
  \end{enumerate}
  In fact since
  \begin{enumerate}
\item  either $\sigma (K_{\lambda}^{+})$ is internally tangent to $\partial \Omega$ at a point $y_0, y_0 \notin T_{\lambda};$
\item or $T_{\lambda}$ is orthogonal to $\partial \Omega$ at  a point $x_0, $
\end{enumerate}
 it suffices    to  just consider these two cases in turn and end up with a contradiction.\\
 Let us suppose at first that the above first case. We have  $u-v>0$ in $ \Sigma_{\lambda_0}\backslash K$ and   there is $x_0 \in \partial \Omega,  x_0 \notin T_{\lambda_0}$ such that  $(u-v)(x_0)= 0.$ Then, thanks to Hopf's lemma, we get $\frac{\partial}{\partial \nu} (u-v)(x_0)> 0,$ that is never but  $c>c.$ This is a contradiction.\\
 Let us suppose this time the second case (case 2). Since we still   have  $u-v>0$ in $ \Sigma_{\lambda_0}\backslash K.$ Following  the same reasonning as in \cite{Ser}, we get  
 $\frac{\partial}{\partial \nu} (u-v)(x_0)= 0$ and   $\frac{\partial^2}{\partial \nu^2} (u-v)(x_0)= 0.$ This  contradicts the Lemma \ref{lemse}.\\
 Therefore, there is only the situation of $\lambda_0= 0.$ And therefore $u \equiv v$ and $T_{\lambda_0}= T_0.$
\end{proof}

\section{The first result of the Schiffer's problem}
We begin this section by recalling  the following  classical but fundamental result for our study:
\begin{theorem} 
Let $\O$ be  a boundomain of class $\mathcal C^{2, \alpha}.$
Let us consider the following boundary value problem
\begin{eqnarray}\label{eq201}
\left\{
\begin{array}{rclcl}
\D u  + k^2 u & = & 0 &\mbox{in}& \O \\
u & = & 0 &\mbox{on} & \dpa \O.
\end{array}
\right.
\end{eqnarray}
Then we have only of the following result:
\begin{itemize}
\item  $(\ref{eq201})$ has $u \equiv 0$ as the unique solution;
\item $(\ref{eq201})$ has non trivial solutions  which which form a finite-dimensional vector subspace of $\mathcal C^{2, \alpha} (\overline{\Omega}).$
\end{itemize}
\end{theorem}
See for instance  \cite{GT}, Theorem 6.15 for more details.

\begin{theorem}\label{DS} 

Let $k\neq 0$  be a given real number.  If $k^2$ is an  eigenvalue of the  Laplacian- Dirichlet operator, then there exist $R>0, \beta (R)>0$ and $\Omega \subset \mathbb{R}^2$ a bounded convex domain  such that 
\begin{eqnarray*}
\begin{cases} 
\Delta u + k^2 u= 0 \; \; in\; \;   \Omega\\
-  \frac{\partial u}{\partial \nu} \big\arrowvert_{\partial \Omega}=  c_0\\
u\big \arrowvert_{\partial \Omega}= 0 \\

\end{cases}
\end{eqnarray*}
has a solution  for
  constants  $c_0> \beta (R).$\\
And as consequence, in this case  the Schiffer's conjecture is true ($\Omega$ is a disc). 
\end{theorem}
\subsection{Existence of minimum of shape functional}
In this subsection we are going to study the following  shape optimization problem

\begin{eqnarray*}
\min_{\omega \in \mathcal O} J(\omega); J(\omega):= \int_{\omega} |\nabla u|^2 dx
\end{eqnarray*}
constrained by the boundary value problem:
\begin{eqnarray*}
\begin{cases} 
\Delta u + k^2 u= 0 \; \; in\; \;   \omega\\
u\big \arrowvert_{\partial \omega}= 0 \\
\displaystyle \int_{\omega} u^2 dx = 1

\end{cases}
\end{eqnarray*}
where $k$ is such that, $|k|< \frac{j_{n, 1}}{R}, R> 0,$ and $c_0$ is a given  positive constant;

 $\mathcal O$ stands for a topological set. Namely, in our work, let $V_0>0$ be a chosen positive real value and $B$ a ball of $\mathbb{R}^N$, then 
$$\mathcal O:= \{\Omega \subset B, \; \; \mbox{open convex set of }\; \; \mathbb{R}^N (N\geq 2), \; \; \mbox{of class }\, \; \mathcal C^m, m \geq 2: |\Omega|= V_0 \}.$$
Let us  introduce the following notation $k^2:= \lambda_{\omega}.$
It is well known that the above shape optimization admits a solution. The reader interested can see  for instance the following references  \cite{BucBut}, \cite{HP}, \cite{DeZo}.\\
We quote also that $\lambda_{\omega_n}$ converges to $ \lambda_{\omega}$ and 
let us call  by $\Omega$ the  minimum domain  of $J$ under the above boundary value problem. \\
We shall use the shape calculus to compute the shape derivative and obtain the  optimality condition.

\subsection{Proof of the Theorem \ref{DS}}
The proof of the theorem is organized in several steps. After the shape optimization part,  we are going to compute  optimality condition and establish a monotonicity result and give the last part of the proof.
\subsubsection{Optimality condition:}
$ J(\Omega):= \displaystyle  \int_{\Omega} |\nabla u|^2 dx, $
\begin{eqnarray*}
\begin{cases} 
\Delta u + k^2 u= 0 \; \; in\; \;   \Omega\\
u\big \arrowvert_{\partial \Omega}= 0 \\
\displaystyle \int_{\Omega} u^2 dx = 1

\end{cases}
\end{eqnarray*}
The first formulae of the shape derivative is given by:
\begin{eqnarray*}
dJ (\Omega, V)= \int_{\Omega} \nabla u \nabla u' dx + \int_{\partial \Omega} \frac{1}{2}|\nabla u|^2 V(0). \nu d\sigma,
\end{eqnarray*}
where  $\nu$ is the exterior normal of $\Omega,$ $ V (t, x); V(0):= V(0, x):= V$ is  a vector field, $t\in [0, \epsilon), \epsilon <<1,$ $u'$  stands for the shape derivative of $u$  and is solution of the following boundary value problem
\begin{eqnarray*}
\begin{cases} 
\Delta u' + k^2 u' + (k^2)' u= 0 \; \; in\; \;   \Omega\\
u'\big \arrowvert_{\partial \Omega}= - \frac{\partial u}{\partial \nu} V(0). \nu \\
\displaystyle \int_{\Omega} u' u dx = 0.

\end{cases}
\end{eqnarray*}
Writing  the variational expression of the Dirichlet eigenvalue problem  with $u'$, we have: 
\begin{eqnarray*}
\int_{\Omega} \nabla u. \nabla u' dx= -\int_{\partial \Omega} (\frac{\partial u}{\partial \nu})^2 V(0). \nu d\sigma.
\end{eqnarray*}
Finally, we get 
\begin{eqnarray*}
dJ(\Omega, V)= -\frac{1}{2}\int_{\partial \Omega} (\frac{\partial u}{\partial \nu})^2 V(0). \nu d\sigma.
\end{eqnarray*}
 At the optimal state, there is a Lagrange multiplier $\tau_{\Omega}$ such that $ -\frac{1}{2}(\frac{\partial u}{\partial \nu})^2= \tau_{\Omega}$ on $\partial \Omega.$

We have the monotonicity result related to the Lagrange multiplier from the optimality condition of the shape optimization problem.
\begin{proposition} 
The map
$\Omega \mapsto \Lambda_{\Omega}:=  (-2\tau_{\Omega })^{1/2}$  is decreasing in the sense that:\\
For any $   \Omega _1, \Omega_2 \, \, \mbox{two starshaped with respect to an origin point}\, \,  0,$  solutions of  the considered shape optimization problem,, $  \overline{\Omega_ 1} \subset \overline{\Omega_2} ,$ then\, $ \Lambda_{\Omega_2} \leq \Lambda_{\Omega_1}.$
\end{proposition}
\begin{proof}
 Let $ \overline{\Omega_ 1} \subset \overline{\Omega_2},$ then  $\exists t^* \in (0, 1), t^* \Omega_2 \subset \Omega_1$ and $\partial (t^* \Omega_2)\cap \partial \Omega_1 \neq \emptyset.$ \\
 Let us set $ \Omega^* = t^* \Omega_2, u_* (x):= u_2 (\frac{x}{t^*}),$ then $\frac{x}{t^*}\in \Omega_2.$ We heve
 \begin{eqnarray*}
\begin{cases} 
\Delta u_* + k^2 u_* = 0 \; \; in\; \;   \Omega^*\\
u_*\big \arrowvert_{\partial \Omega^*}= 0

\end{cases}
\end{eqnarray*}
Since $t^* \Omega_2 \subset \Omega_1,$ let us set $w= u_1$ defined in $\Omega^*.$
Let us set $v= u_1- u_*$
 \begin{eqnarray*}
\begin{cases} 
\Delta v + k^2  v = 0 \; \; in\; \;   \Omega^*\\
v\big \arrowvert_{\partial \Omega^*}= u_1

\end{cases}
\end{eqnarray*}
SInce $v= u_1\geq 0$ on $\partial \Omega^*,$ by maximum principle we have $v= u_1- u_* \geq 0$ in $\Omega^*.$ 
Let $x^* \in \partial \Omega^* \cap \partial \Omega_1,$ for $h>0$ small enough, we have $u_1 (x^*-h \nu)- u(x^*)\geq  u_* (x^*-h \nu)-u_* (x^*).$
This implies that  $-\frac{\partial u_1}{\partial \nu} (x^*)\geq  -\frac{\partial u_*}{\partial \nu} (x^*).$ This yields $(-2\tau_{\Omega_1})^{1/2}\geq \frac{1}{t^*}(-2\tau_{\Omega_2})^{1/2}> (-2\tau_{\Omega_2})^{1/2}$ and  finally we have $\Lambda_{\Omega_1}> \Lambda_{\Omega_2}.$
\end{proof}


\subsubsection{ Last part of the proof of Theorem \ref{DS}}
\begin{proof}
Let us take  a convex  domain $\Omega_0 \subset B(0, R):= B_R,$ where $R$ is chosen as follows: 
\begin{eqnarray*}
k \neq 0, 0< R < \frac{j_{n, 1}}{|k|}.
\end{eqnarray*}
 Let us first consider $u_R$ be the solution of the Dirichlet eigenvalue problem in $B_R.$ Since $\Omega_0$ is a convex domain,  there is $t_1 \in (0, 1), t_1 B_R \subset \Omega_0$ and $\partial \Omega_0 \cap t_1 \partial B_R \neq \emptyset.$\\
Let us set $u_{t_1} (x)= u_R (\frac{x}{t_1}),  \frac{x}{t_1}\in B_R.$\\
We have 
\begin{eqnarray*}
(-2\tau_{\Omega})^{1/2}> \|\nabla u_R (\frac{x_1}{t_1})\|, \, \, \mbox{for}\, \, x_1\in  \partial \Omega \cap t_1\partial B_R.
\end{eqnarray*}
As initialization we  choose $\Omega_0,$  and then we can build a sequence $(\Omega_n)_{n \in \mathbb{N}}: \cdots \Omega_2 \subset \Omega_1\subset \Omega_0 \subset B_R$ which generates a decreasing sequences $(\Lambda_{\Omega_{n}}), n\in \mathbb{N}$ in the sense:
\begin{eqnarray*}
\forall n \in  \mathbb{N}, \bar{\Omega}_{n+1} \subset \bar{\Omega}_{n}  \Rightarrow  
\Lambda_{\Omega_{n+1}} >  \Lambda_{\Omega_{n}}.
\end{eqnarray*}
\begin{eqnarray*}
\forall n\in \mathbb{N}:  \bar{\Omega}_{n+1} \subset \bar{\Omega}_{n}  \Rightarrow  
\Lambda_{\Omega_{n+1}} >  \Lambda_{\Omega_{n}}> \|\nabla u_R\|\big \arrowvert_{\partial B_R}.
\end{eqnarray*}
It suffices to take $ \beta(R):=   \|\nabla u_R\|\big \arrowvert_{\partial B_R}$. And then, by approximation we have:  for $c_0>   \|\nabla u_R\|\big \arrowvert_{\partial B_R}$ $\exists \Omega^*\in \mathcal O:$
\begin{eqnarray*}
\begin{cases} 
\Delta u_* + k^2 u_* = 0 \; \; in\; \;   \Omega^*\\
u_*\big \arrowvert_{\partial \Omega^*}= 0\\
-\frac{\partial u_*}{\partial \nu }\big \arrowvert_{\partial \Omega^*}= c_0
\end{cases}
\end{eqnarray*}
By the maximum principle and symmetry Serrin's method,  we conclude that $\Omega^*$ is a disc.
\end{proof}
	\begin{remark} 
Here after, we give  the  directional shape derivative of $k^2= \lambda (\Omega)$  when it is  a multiple eigenvalue.  In fact there is no differentiability but only directional derivative.  
	 	\begin{theorem}\textnormal{(Derivative of a multiple Dirichlet eigenvalue)}\label{thm}\\
		Let $\Omega$ be a bounded open domain of $\mathbb{R}^N$ of class $\mathcal C^2$. Suppose $\lambda_k(\Omega)$ is an eigenvalue of multiplicity $p>2$. Let $u_{k_1},u_{k_2},\ldots,u_{k_p}$ be an orthonormal family (for the scalar product $L^2$) of eigenvalues associated with $\lambda_k$. Then $t\mapsto\lambda_k(\Omega_t)$ has a directional derivative at $t=0$ which is one of the eigenvalues of the matrix $p\times p$ defined by
		\begin{align}
			\mathcal{M}=(m_{i,j})\;\;\text{with}\;\;m_{i,j} =-\int_{\partial\Omega}\left(\frac{\partial u_{k_i}}{\partial \nu }\frac{\partial u_{k_j}}{\partial \nu}\right) V\cdot \nu \;d\sigma\;\;i,j=1,\ldots,p
		\end{align}
		where $\frac{\partial u_{k_i}}{\partial \nu}$ is the normal derivative of the $k_i$-th eigenfunction $u_{k_i}$ and $V\cdot \nu$ is the normal displacement of the boundary induced by the deformation field $V$.
	\end{theorem}
	The proof of this theorem can be found in  \cite{berger2015optimisation}.  Alexandre Munnier gave a matrix demonstration of the theorem in his doctoral thesis \cite{munnier2000stabilite}. The first work to our knowledge is by Bernard Rousselet \cite{rousselet1983shape} in his  study of the static response and eigenvalues of a membrane as a function of its shape.
		\end{remark}
	
	
\section{Pompeiu's Problem}
In this section we  would like to discuss the Pompeiu's conjecture based on the monotonicity property established in the optimality condition. But the the clogging relies on first the application of the maximum principle. And secondly, even if  we are in the valid situation of the maximum principle, i.e. when $0 < |k|< \frac{j_{n, 1}}{R},$ the increase of the radius  $R$ of the ball $B_R$ implies the decrease of $k$  which is chosen and fixed in advance. And we may not  get the condition  $\frac{\partial u}{\partial \nu}= 0$ on $\partial \Omega.$\\
In what follows, we are going to propose another way to  address this questions by introducing aspects of infinite-dimensional Riemannian geometry which are combined with shape derivative and shape Hessian. \\
Before that let us begin by giving some classical shape derivative calculus of the Neumann eigenvalue problem recalled below:
\begin{equation*}
			\begin{cases}
				\Delta u +\lambda (\Omega) u=0\;\;\text{in}\;\;\Omega\;\;\text{with}\;\;\lambda(\Omega)=k^2\\[0.3cm] \frac{\partial u}{\partial \nu}=0\;\;\text{on} \;\;\partial\Omega.
			\end{cases}
		\end{equation*} 
\subsection{Shape derivative of Neumann eigenvalues problem}
We start this subsection by recalling the definition of tangential gradient and divergence see for instance (\cite{ms}, \cite{s}, \cite{HP}, \cite{sozo}, \cite{DeZo}).
\begin{definition}
 Let $\Omega$ be a given domain with the boundary  $\Gamma= \partial \Omega$ of class $\mathcal C^2,$ and $V\in \mathcal C^1(U; \mathbb{R}^N)$ be a vector field; $U$ be  an open neighborhood of the manifold $\Gamma \in \mathbb{R}^N$. Then the following notation is used to define the tangential divergence as:
 \begin{eqnarray*}
 div_{\Gamma}V = (divV -  \langle DV  \nu, \nu\rangle_{\mathbb{R}^N})\vert _{\Gamma} \in \mathcal C(U).
 \end{eqnarray*}
Let $\Omega$ be a given domain with the boundary  $\Gamma= \partial \Omega$ of class $\mathcal C^2,$ and $V\in \mathcal C^1(\Gamma; \mathbb{R}^N)$ be a vector field. The tangential divergence of $V$ on $\Gamma$ is given by:
 \begin{eqnarray*}
 div_{\Gamma}V = (div \tilde V -  \langle D\tilde V  \nu, \nu\rangle_{\mathbb{R}^N})\vert _{\Gamma} \in \mathcal C(\Gamma).
 \end{eqnarray*}
 where $\tilde V $is any $C^1$ extension of $V$ to an open neighborhood of $\Gamma \subset \mathbb{R}^N$
  and
   $$D\tilde V= ((D\tilde V)_{ij})_{i,j \in \{1, \cdots N\}}, (D\tilde V)_{ij}:= \frac{\partial \tilde V_i}{\partial x_j}.$$ 
\end{definition} 
The notion of tangential gradient  $\nabla_{\Gamma}$ on $\Gamma$
\begin{eqnarray*}
\nabla_\Gamma:  \mathcal C^2 (\Gamma) \rightarrow \mathcal C^1 (\Gamma, \mathbb{R}^N)
\end{eqnarray*}
is defined as follows
\begin{definition}
Let  $ h\in \mathcal C^2(\Gamma)$ be given and let $\tilde h$ be an extension of $h,\tilde h \in \mathcal C^2(U)$ and $\tilde h_{\vert \Gamma}=h;$ $U$ be an open neighborhood of $ \Gamma$ in $\mathbb{R}^ N$. Then
\begin{eqnarray*}
\nabla_\Gamma h =  {\nabla \tilde h}_{\vert \Gamma} - \frac{\partial \tilde h}{\partial \nu} \nu.
\end{eqnarray*}
\end{definition}
\begin{theorem} (Derivative of a simple Neumann eigenvalue)\\
		Let $\Omega$ be a bounded open of $\mathbb{R}^N$ of class $\mathcal C^2$. Suppose $\lambda(\Omega)$ is a simple eigenvalue and $u=u_\Omega$ its associated eigenfunction. Then the functions $t\mapsto \lambda(t)=\lambda(\Omega_t)$, $t\mapsto u (\Omega_t)$ are differentiable at $t=0$. The derivative of the eigenvalue is given by
		\begin{eqnarray}\label{val2}
			\lambda'(0)=\int_{\partial\Omega}\left(\lvert\nabla u \rvert^2-\lambda u^2\right)(V\cdot \nu)\;d\sigma
		\end{eqnarray}
		and the derivative $u'$ of $u_t= u(\Omega_t)$ is a solution of 
		\begin{equation*}
			\begin{cases}
				-\Delta u  ' =\lambda'(0) u +\lambda u ' \;\;\;\text{in}\;\;\Omega \\[0.2cm]  \frac{\partial u'}{\partial \nu}=\left(-\frac{\partial^2 u}{\partial \nu^2}\right) V\cdot \nu+\nabla u\cdot\nabla_\Gamma(V\cdot \nu)\;\;\;\text{on}\;\;{\partial\Omega} \\[0.2cm]  \displaystyle \int_{\partial\Omega}u^2(V\cdot \nu)\;d\sigma+2\int_\Omega u u' \;dx=0.
			\end{cases}
		\end{equation*}
		$\nabla_\Gamma$ is the tangential gradient.
	\end{theorem} 
	
	\begin{theorem}\textnormal{(Derivative of a multiple Neumann eigenvalue)}\\
		Let $\Omega$ be a bounded open of class $\mathcal{C}^2$. Suppose $\lambda_k(\Omega)$ is an eigenvalue of multiplicity $p\geq 2$. Let $u_{k_1},u_{k_2},\ldots,u_{k_p}$ be an orthonormal family (for the scalar product $L^2$) of eigenvalues associated with $\lambda_k$. Then $t\mapsto\lambda_k(\Omega_t)$ has a directional derivative at $t=0$ which is one of the eigenvalues of the matrix $p\times p$ defined by
		\begin{align}
			\mathcal{M}=(m_{i,j})\,\,\text{with}\,\,m_{i,j} =\int_{\partial\Omega}(\nabla u_{k_i}\cdot\nabla u_{k_j})\;d\sigma -k^2\int_{\partial\Omega}u_{k_i} u_{k_j} (V\cdot \nu )\;d\sigma\;\;i,j=1,\ldots,p
		\end{align}
		where $V\cdot \nu$ is the normal displacement of the boundary induced by the deformation field $V$.
	\end{theorem}
	For the proof,  the  computation techniques  used  in  the Dirichlet case lead us to get the desired result.
\subsection{ Riemannian geometry and  sufficient conditions for shapes}
The aim is to analyze the  correlation of the  Riemannian geometry on called  infinite dimensional manifolds $B_{e}$ with shape optimization.\\
The author would like to stress, what follows has been already done in pioneering works, see \cite{mm1},  \cite{mm2},  \cite{mm3}. Let us  reproduce some fundamental steps related to our work.\\

 Let   $\Omega$ be a simply connected and compact subset of $\mathbb{R}^{2}$ with $\Omega\neq\emptyset$ and $\mathcal{C}^{\infty}$ boundary $\partial\Omega$. As is always the case in shape optimization, the boundary of the shape is all that matters. Thus we can identify the set of all shapes with the set of all those boundaries.\\
 
  Let $Emb(\mathbb{S}^{1},\mathbb{R}^{2})$ be the set of all  smooth embeddings on $\mathbb{S}^{1}$ in the plan $\mathbb{R}^{2}$, its elements are the injective mappings $c:\mathbb{S}^{1}\longrightarrow\mathbb{R}^{2}$. Let $Diff(\mathbb{S}^{1})$ stands for the set of all $\mathcal{C}^{\infty}$ diffeomorphism on $\mathbb{S}^{1}$ which acts diferentiably on $Emb(\mathbb{S}^{1},\mathbb{R}^{2}).$ Let us consider $B_{e}$ as the quotient $Emb(\mathbb{S}^{1},\mathbb{R}^{2})\slash Diff(\mathbb{S}^{1})$. In terms of sets, we have 
$$
B_{e}(\mathbb{S}^{1},\mathbb{R}^{2}):=\{ \ [c]\ /\  c\in Emb \}\ 
\mbox{where} \  [c]:=\{c^{\prime}\in Emb\  /\ c^{\prime}\sim c \}.$$
To characterize the tangent space to $B_{e}$ we start with the characterization of the tangent space to $Emb$ denoted $T_{c}Emb$ and the tangent space to the orbit of $c$ by $Diff(\mathbb{S}^{1})$ at $c$ denoted by $T_{c}(Diff(\mathbb{S}^{1}).c)$. Thus the tangent space to $B_{e}$ is then identified with a supplementary subspace of $T_{c}(Diff(\mathbb{S}^{1}).c)$ in $T_{c}Emb$. 
\begin{proposition}
Let $c\in Emb$, then the tangent space at $c$ to $Emb$ is given by: $T_{c}Emb=\mathcal{C}^{\infty}(\mathbb{S}^{1},\mathbb{R}^{2}).$
\end{proposition}
\begin{proposition}\label{proposition-2}
The tangent space to the orbit of $c$ by $Diff(\mathbb{S}^{1})$, is the subspace of  $T_{c}Emb$ formed by vectors $m(\theta)$ of the type $c_{\theta}(\theta)=c^{'}(\theta)$ times a function.
\end{proposition}
\begin{remark}
The choice of the supplementary must abide by the action  of $Diff(\mathbb{S}^{1})$ i.e we choose a supplementary of $T_{c}(Diff(\mathbb{S}^{1}).c)$ in $T_{c}Emb$ stable by the action of $Diff(\mathbb{S}^{1})$. For that it suffices to define a metric on $Emb$ for which $Diff(\mathbb{S}^{1})$ acts isometrically and define the supplementary of $T_{c}(Diff(\mathbb{S}^{1}).c)$ as its orthogonal with respect to this metric.
\end{remark}
\begin{definition}
Let $G^{0}$ be metric invariant by the action of $Diff(\mathbb{S}^{1})$ on the manifold $Emb(\mathbb{S}^{1},\mathbb{R}^2)$, defined by the application:
$$
\begin{array}{ccccl}
G^{0} & : & T_{c}Emb\times T_{c}Emb & \to & \mathbb{R} \\
 & & (h,m) & \mapsto & \displaystyle \int_{\mathbb{S}^{1}}\big<h(\theta),m(\theta)\big>|c^{\prime}(\theta)|d\theta \\
\end{array}
$$
where $\big<h(\theta),m(\theta)\big>$ is the ordinary scalar product of $h(\theta)$ and $m(\theta)$ in $\mathbb{R}^{2}$. 
\end{definition}
 \begin{proposition}\label{proposition-4}
 Let $c\in B_{e}$ then $T_{c}B_{e}$ is colinear to the outer unit normal of $\Omega$. In other words $$T_{c}B_{e}\simeq\{ h\ |\  h=\alpha \nu,\alpha\in\mathcal{C}^{\infty}(\mathbb{S}^{1},\mathbb{R}) \}.$$
 \end{proposition}

Now let us consider the following terminology:
$$ds=|c_{\theta}|d\theta\qquad\mbox{arc length}. $$
\begin{definition}
A Sobolev-type metric on the manifold $B_{e}(\mathbb{S}^{1},\mathbb{R}^{2})$ is map:
$$
\begin{array}{ccccl}
G^{A} & : & T_{c}B_{e}\times T_{c}B_{e} & \to & \mathbb{R} \\
 & & (h,m) & \mapsto & \displaystyle \int_{\mathbb{S}^{1}}(1+AK_{c}^{2}(\theta))\big<h(\theta),m(\theta)\big>|c^{\prime}(\theta)|d\theta \\
\end{array}
$$
where $K_{c}$ is the  curvature of $c$ and $A$ a positive real. 
\end{definition}
\begin{remark}
\begin{enumerate}
\item By setting $h=\alpha \nu$, $m=\beta \nu$ and by parametrizing $c(s)$ by arc length we have :
$$
G^{A}(h,m)=\int_{\partial\Omega}(1+AK_{c}^{2}(\theta))\alpha\beta ds.
$$
\item If $A> 0,$ $G^{A}$ is a Riemannian metric.
\end{enumerate}
\end{remark}
Before proceeding further, let us define the first Sobolev metric which  generalize the above  Riemannian metric and does not induce the phenomenon of vanishing geodesic distance studied in \cite{mm2}.
\begin{definition}
The first Sobolev metric on $B_{e}(S^1, \mathbb{R}^2)$ is given by
\begin{eqnarray*}
g: T_c (B_{e}(S^1, \mathbb{R}^2))\times T_c( B_{e}(S^1, \mathbb{R}^2)) \rightarrow \mathbb{R},\\
 (h, k) \mapsto \int_{S^1}\langle (I- AD_s^2)h, k\rangle ds,
\end{eqnarray*}
where $A> 0$ and $D_s$ denotes the arc length derivatives with respect to $c$ defined by:\\ $D_s:= \frac{\partial_{\theta}}{|c_{\theta}|}, c_{\theta}= \frac{\partial c}{\partial \theta}, ds= |c_{\theta}| d\theta.$
\end{definition}
In Riemannian geometry it is important to have a good understanding of the so called  covariant derivative which is and operation involving in the differential calculus in differential geometry. In  what we are going to discuss in next section, the expression of covariant derivative appears in Riemannian shape Hessian. And its computation  becomes a key step.
Let us reproduce the following theorem  due  to Welker (cf \cite{W} for more details).
\begin{theorem}
Let $A>0$ and let $h, m \in T_c Emb(S^1, \mathbb{R}^2)$ denote vector field along $c \in Emb(S^1, \mathbb{R}^2). $ The arc length derivative with respect to $c$ is denoted by $D_s.$ Moreover, $L_1:= I-AD_s^2$ is a differential operator on $\mathcal C^{\infty}(S^1, \mathbb{R}^2)$ and $L_1^{-1}$ denotes the inverse operator. The covariant derivative associated with the Sobolev metric $g$ can be expressed as
\begin{eqnarray*}
\nabla_m h= L_1^{-1}(K_1 (h))\; \; \mbox{with}\; \; K_1:= \frac{1}{2}\langle D_s m, v\rangle (I+ A D_s^2),
\end{eqnarray*}
where $v= \frac{c_{\theta}}{|c_{\theta}|}$ denotes the unit tangent vector.
\end{theorem}
\subsection{Shape derivative  of first order and covariant derivative}
One  considers the following  constrained  shape optimization problem
\begin{eqnarray*}
\min_{\Omega \in \mathcal O} J_ 2 (\Omega), \\
J_2 (\Omega):= \int_{\partial \Omega} u^2 d\sigma + \gamma \; \; vol(\Omega), \; \; \mbox{where}\, \, \gamma < 0, \, \; vol(\Omega)= \int_{\Omega} dx.
\end{eqnarray*}
with
\begin{equation*}
			\begin{cases}
				\Delta u +\lambda(\Omega)u =0\;\;\text{in}\;\;\Omega\\\frac{\partial u}{\partial \nu}=0\;\;\text{on} \;\;\partial\Omega\\
				\int_{\Omega} u^2 dx= 1
			\end{cases}
		\end{equation*} 
		 with $\mathcal O$ standing  for a topological set recalled here after. Let $V_0>0$ be a chosen positive real value and $B$ a ball of $\mathbb{R}^N$, then 
$$\mathcal O:= \{\Omega \subset B, \; \; \mbox{open convex set of }\; \; \mathbb{R}^N (N\geq 2), \; \; \mbox{of class }\, \; \mathcal C^m, m \geq 2: |\Omega|= V_0 \}.$$
		
		To compute the shape derivative we use   a classical formulae which can be found in \cite{HP} (Proposition 5.4.18, pp 225). It is given by
		\begin{eqnarray*}
		dJ_2 (\Omega)[V]:= dJ_2 (\Omega, V)= \int_{\partial \Omega} 2 u u' + (V. \nu) [\frac{\partial u^2}{\partial \nu} + H u^2]+ \gamma  V. \nu
		\end{eqnarray*}
		where $H$ is the mean curvature of $\partial \Omega$ and $u'=\frac{d u_t}{dt}_{\vert t= 0}= \frac{du_{\Omega_t}}{dt}_{\vert t= 0} $ is the shape derivative associated to the Laplace-Neumann eigenvalue problem. 
		\\ Let us note that in two dimension $H= K_c.$\\
If  we look for $u'$ such that $u'= - \frac{ \partial u}{\partial \nu} V(0). \nu,$  then we have:\\
the material derivative, called also Lagrange derivative
$\dot u_{\vert \partial \Omega}= \frac{d}{dt}_{\vert t= 0}(u_t \circ T_t)= 0, 0< t < \epsilon< 1, $ very small; $\Omega_t:= T_t (\Omega), \{T_t\}_{t}$ a familly of diffeomorphisms. 

So  the first  shape derivative of $J_2$   is reduced  as follows:
\begin{equation*}
			\begin{cases}
				dJ_2 (\Omega, V)= \int_{\partial \Omega}  (V. \nu) [ H u^2+ \gamma ]\\
				\dot u_{\vert \partial \Omega}= 0\\
				2 \int_{\Omega} u u' dx + \int_{\partial \Omega} u^2 V. \nu d\sigma = 0
			\end{cases}
		\end{equation*} 
		with recalling that 	$div_{\partial \Omega} \nu = H $ is the  tangential divergence.
		
		In the sequel, we are going to keep in mind the condition $\dot u = 0$ on $\partial \Omega .$  This one will play a key role to get the Dirichlet condition in the Pompeiu problem.

If $V_{|\partial\Omega}=\alpha \nu$ we can still write :
\begin{eqnarray}\label{GRD1}
dJ_2(\Omega)[V]=\int_{\partial\Omega}\left(H u^2 + \gamma  \right)\alpha d\sigma.
\end{eqnarray}
It should be noted that there is a link between the shape derivative of $J_2$ and the gradient in Riemannian structures see \cite{Schu} and \cite{W}. To illustrate our claim, let us consider  the Sobolev metric $G^A$ to ease the understanding of the computations. We think that it is quite possible  to generalize this  study in higher dimensions and even with other metrics.\\
Our purpose is to  calculate the gradient of $J_2: B_{e}\to\mathbb{R}$ then we have :
\begin{equation}\label{GRD2}
dJ(\Omega)[V]=G^{A}(grad J_2(\Omega), V)
\end{equation}
if $V_{|\partial\Omega}=h$ we have
\begin{eqnarray}
dJ_{c}(h)&=&G^{A}(grad J_2(\Omega), h)\nonumber\\
dJ_{c}(h)&=&\int_{\partial\Omega}\left(1+AK^{2}_{c}\right)gradJ_2\alpha. \nonumber
\end{eqnarray}
But from (\ref{GRD2}), $$dJ_{c}(h)=\int_{\partial\Omega}\left(H u^2 + \gamma  \right)\alpha d\sigma$$ and thus 
$$
\int_{\partial\Omega}\left(H u^2 + \gamma  \right)\alpha d\sigma =\int_{\partial\Omega}\left(1+AK^{2}_{c}\right)gradJ_2\alpha d\sigma
$$
so that $$gradJ_2=\frac{1}{1+AK^{2}_{c}}\left(H u^2 + \gamma  \right).$$

The next step is  to compute the explicit form of the covariant derivative $\nabla_{h}m\in T_{c}B_{e}$ with $h, m\in T_{c}B_{e}.$\\
 The following result has been established first in a pioneering work (see \cite{Schu}), and  for additional details, see \cite{DS1}.
\begin{theorem}
Let $\Omega\subset\mathbb{R}^{2}$ be at least of class $\mathcal C ^2,$ $V, W\in\mathcal{C}^{\infty}(\mathbb{R}^{2},\mathbb{R}^{2})$ vector fields which are orthogonal to the boundaries i.e  $$V_{|\partial\Omega}=\alpha \nu$$ with $\alpha:=\big<V_{|\partial\Omega},\nu\big>$ and $$W_{|\partial\Omega}=\beta\nu$$ with $\beta:=\big<W_{|\partial\Omega},\nu\big>$ such that $V_{|\partial\Omega}=h:=\alpha \nu $,\ $W_{|\partial\Omega}=m=:\beta \nu$ belongs to the tangent space of $B_{e}$. Then the covariant derivative associated with the Riemannian metric $G^{A}$ can be expressed as follows:
\begin{eqnarray}
\nabla_V{W}:&=&\nabla_{h}{m}=\frac{\partial\beta}{\partial \nu }\alpha+\left(\frac{3AK_{c}^{3}+K_{c}}{1+AK_{c}^{2}}\right)\alpha\beta\nonumber\\
 &=&\big<D_{V}W, \nu\big>+\left(\frac{3AK_{c}^{3}+K_{c}}{1+AK_{c}^{2}}\right)\big<V, \nu \big>\big<W, \nu \big>\nonumber.
\end{eqnarray}
where $D_{V}W$ is the directional derivative of the vector field $W$ in the direction $V$.
\end{theorem}
See\cite{DS1} for the details of the proof.
\begin{remark}
 Let us now  calculate the torsion of the connection $\nabla.$
Indeed, one is wondering if   the connection $\nabla$  coincides with the Levi-Civita connection.\\
 We have 
\begin{eqnarray}
T(V,W)&=&\nabla_{V}W-\nabla_{W}V-[V,W]\nonumber\\
T(V,W)&=&\big<D_{V}W,\nu\big>+\left(\frac{3AK^{3}_{c}+K_{c}}{1+AK_{c}^{2}}\right)\big<V, \nu \big>\big<W, \nu\big>\nonumber\\
          &-&\big<D_{W}V,\nu\big>-\left(\frac{3AK^{3}_{c}+K_{c}}{1+AK_{c}^{2}}\right)\big<V, \nu \big>\big<W, \nu \big>-[V,W]\nonumber\\
T(V,W)&=&\frac{\partial\beta}{\partial \nu }\alpha-\frac{\partial\alpha}{\partial \nu }\beta-[h,m].\nonumber
\end{eqnarray}
But $$\frac{\partial\beta}{\partial \nu }\alpha-\frac{\partial\alpha}{\partial \nu}\beta=[h,m].$$
Then we have:
\begin{eqnarray}
T(V,W)&=&[h,m]-[h,m]\nonumber\\
T(V,W)&=&0.\nonumber
\end{eqnarray}
As a conslusion, we claim that
 $\nabla$ is compatible with the metric $G^{A}$ and its torsion is zero, so it coincides with the Levi-Civita connection. 
\end{remark}

\subsection{ Sufficient condition for the minimality of a  shape functional}

 In this section,  assuming  at first that there is at least one critical point, we shall  first present the sufficient condition on the existence of a local minimum  for a functional \ $J(\Omega)$  given as follows:

 \begin{equation}\label{l}
    J(\Omega) = \displaystyle\int_\Omega \ f_0(u_\Omega, \nabla u_\Omega)
 \end{equation}
 where $f_0$ is a function of $\mathbb{R} \times \mathbb{R}^n$ that we suppose to be smooth and $u_\Omega$ denotes a smooth solution of a boundary value problem.\\
 And in the second part,  in the case of   $J_2(\Omega), $ we compute the second shape derivative.
 
 The fundamental question is then to study  the existence of  the local strict  minima of this functional under possible constraints that $\Omega$ is a critical point. That means that the first order derivative with respect to the domain is equal to zero at the domain $\Omega.$ We shall examine, for that, how  this solution $u_{\Omega}$ varies when its domain of definition $\Omega$ moves.\\
 
 Let us recall the classic method of studying a critical point. Let $(B, \ \| \ . \ \|_1)$ be a Banach space and let $E : (B, \| \ . \ \|_1) \longrightarrow \mathbb{R}$ be a function of class $\mathcal C^2$ whose differential $Df$ vanishes at $0$. The Taylor-Young formula is then written as 
 \begin{equation}\label{TY}
    E(u) = E(0) + D^2 \ E(0) \ . \ (u , u) + o(|| u
    ||_1^2).
 \end{equation}
 In particular, if the Hessian form $D^2 E(0)$ is coercive in the norm $\| \ . \ \|_1$, then the critical point  $0$ is a strict local  minimum of $E$. The fundamental difficulty in the study of critical forms is caused by  the appearance of a second norm $\| \ . \ \|_2$ finer than $\| \ . \ \|_1$ \ \ $(i.e \  \ \| \ . \ \|_2 \leq C \| \ . \ \|_1)$. The Hessian form, is not in general, coercive for the norm $\| \ . \ \|_1$ but it is for the standard norm $\| \ . \ \|_2$. If these norms are not equivalent, which is generally the case, concluding that the minimum is strict is impossible, even locally for the strong norm.  For  an illustartation, cf \cite{DS1}.
 
In the case where $\O_0$  is  a critical point  for  the functional $J,$ to show that  it is a strict local minimum,
 we have  
to   study  the positiveness of a quadratic form  which  is obtained by  computing  the second derivative of
$J$ with respect to the domain. So before proceeding further, we need some hypotheses ;\\
 let us suppose that:
 \begin{description}
 \item{(i)\,\,-} $\O$ is a ${\mathcal C}^{2}-$ regular \,\, open domain.
 \item{(ii)\,\,-} $V(t, x) = \alpha(x) \nu(x),\,\,\,\alpha\,\,\in\,\,H^{\frac{1}{2}}(\dpa \O),\,\,\forall\,\,t\,\,\,\in\,\,[0,\epsilon[$.
 \end{description}
 In \cite{DaPi}, (see also \cite{Da1}, \cite{Da2}), the authors showed that it is not sufficient
 to prove that the quadratic form is positive to claim that: a critical
 shape is a minimum.
 In fact most of the time people use the Taylor Young formula to study
 the positiveness of the quadratic form.\\
 For $t \in [0, \epsilon[$, $j(t):= J(\O_t)=J(\O)+ t dJ(\O, V)+ \frac{1}{2}t^2
 d^2J(\O,V,V)+o(t^2), V= V(0, x)= V(0).$\\
 The quantity $o(t^2)$ is expressed with the norm of $\mathcal
 C^2$. The $H^{\frac{1}{2}}(\dpa \O)$ norm appears in the expression of $d^2J(\O,V,V)$. And these two norms are not equivalent.
 The quantity $o(t^2)$ is not smaller than $||V||_{H^{\frac{1}{2}}(\dpa
 \O)}$, see the example in \cite{DaPi}.
 Then such an argument does not insure that the critical point is
 a local strict minimum.\\
 
 In our study, we shall use  the hessian obtained via the Sobolev metric $G^A.$ 
\subsection{Positiveness of the quadratic form in the infinite Riemannian point of view}
\begin{definition}
Let $J: \Omega\to\mathbb{R}$ be an functional. One defines the hessian Riemannian shape as follows:
$$
Hess J(\Omega)[V]:=\nabla_{V}grad J 
$$
where $\nabla_{V}$ denotes the derivative following the vector field $V$.
\end{definition}
\begin{theorem}\label{punk0}
The hessian Riemannian shape defined by the Riemannian metric $G^{A}$ verifies the following condition:
$$
G^{A}(Hess J(\Omega)[V],W)=d(dJ(\Omega)[W])[V]-dJ(\Omega)[\nabla_{V}W].
$$
\end{theorem}
\begin{proof}
Our purpose is to show that 
$$
G^{A}(Hess J(\Omega)[V],W)=d(dJ(\Omega)[W])[V]-dJ(\Omega)[\nabla_{V}W].
$$
So  let us  use the compatibility of the metric $G^{A}$ with the Levi-Civita  connection. We have
\begin{eqnarray}
V.G^{A}(grad J,W)&=&G^{A}(grad J,\nabla_{V}{W})+G^{A}(\nabla_{V}{grad J},W),\nonumber\\
G^{A}(\nabla_{V}{grad J},W)&=&V.G^{A}(grad J, W)-G^{A}(grad J,\nabla_{V}W)\nonumber.
\end{eqnarray}
Since $G^{A}(Hess J(\Omega)[V], W)=G^{A}(\nabla_{V}grad J , W),$ we have
\begin{eqnarray}
G^{A}(Hess J(\Omega)[V],W)&=&V.G^{A}(grad J,W)-G^{A}(grad J,\nabla_{V}W,\nonumber\\
G^{A}(Hess J(\Omega)[V],W) &=&V.(WJ)-(\nabla_{V}W).J,\nonumber\\
 G^{A}(Hess J(\Omega)[V],W) &=&d(dJ(\Omega)[W])[V]-dJ(\Omega)[\nabla_{V}W]\nonumber
\end{eqnarray}
where  $V,W\in\mathcal{C}^{\infty}(\mathbb{R}^{2},\mathbb{R}^{2})$ are vector fields normal to the boundary $\partial\Omega$ and $d(dJ(\Omega)[W])[V]$ define\textcolor{red}{s} the standard Hessian shape.
\end{proof}

 Let us  compute $G^{A}(Hess J(\Omega)[V],W)$ by using directly the Sobolev-type metric $G^A$.  Then we have the following proposition.
 \begin{proposition}
 \begin{eqnarray}\label{punk2}
G^{A}(Hess J_2(\Omega)[V],W)=
\int_{\partial\Omega}\left[\frac{\partial}{\partial \nu }\left (H u^2 + \gamma \right)+ K_{c}\left(H u^2 + \gamma \right)\right]\big<V, \nu\big>\big<W, \nu \big> d\sigma.
\end{eqnarray}
 \end{proposition}
 \begin{proof}
\begin{eqnarray}
G^{A}(Hess J_2(\Omega)[V],W)&=&\int_{\partial\Omega}\left(1+AK^{2}_{c}\right)Hess J_2(\Omega)[V]W,\nonumber\\
&=&\int_{\partial\Omega}\left(1+AK^{2}_{c}\right)\nabla_{V}gradJ_2(\Omega)W,\nonumber\\
 &=&\int_{\partial\Omega}\left(1+AK^{2}_{c}\right)\nabla_{h}gradJ_2(\Omega)m.\nonumber
\end{eqnarray}
 Since  $gradJ_2(\Omega)=\frac{1}{1+AK^{2}_{c}}\psi, \psi:= Hu^2 + \gamma$, we have
\begin{eqnarray}
\nabla_{h}gradJ_2(\Omega)&=&\frac{\partial}{\partial \nu}\left(gradJ_2(\Omega)\right)\alpha+\left(\frac{3AK^{3}_{c}+K_{c}}{1+AK_{c}^{2}}\right)gradJ_2(\Omega)\alpha,\nonumber\\
&=&\frac{\partial}{\partial \nu}\left(\frac{1}{1+AK^{2}_{c}}\psi\right)\alpha+\frac{1}{1+AK^{2}_{c}}\psi\left(\frac{3AK^{3}_{c}+K_{c}}{1+AK_{c}^{2}}\right)\alpha,\nonumber\\
&=&\frac{\partial}{\partial \nu}\left[(1+AK^{2}_{c})^{-1}\right]\psi\alpha+\frac{\partial\psi}{\partial \nu}\left(\frac{1}{1+AK^{2}_{c}}\right)\alpha+\frac{1}{1+AK^{2}_{c}}\psi\left(\frac{3AK^{3}_{c}+K_{c}}{1+AK_{c}^{2}}\right)\alpha,\nonumber\\
&=&-2AK_{c}\frac{\partial K_{c}}{\partial \nu}\left(1+AK^{2}_{c}\right)^{-2}\psi\alpha+\frac{\partial\psi}{\partial \nu}\left(\frac{1}{1+AK^{2}_{c}}\right)\alpha\nonumber\\&+&\frac{1}{1+AK^{2}_{c}}\psi\left(\frac{3AK^{3}_{c}+K_{c}}{1+AK_{c}^{2}}\right)\alpha\nonumber.
\end{eqnarray}
Note that $\frac{\partial K_{c}}{\partial \nu}= K_c^2,$ (cf \cite{DS1}) which implies that:
\begin{eqnarray}
\nabla_{h}gradJ_2(\Omega)
&=&\frac{-2AK^{3}_{c}}{\left(1+AK^{2}_{c}\right)^{2}}\psi\alpha+\frac{\partial\psi}{\partial \nu}\left(\frac{1}{1+AK^{2}_{c}}\right)\alpha+\frac{1}{1+AK^{2}_{c}}\psi\left(\frac{3AK^{3}_{c}+K_{c}}{1+AK_{c}^{2}}\right)\alpha \nonumber.
\end{eqnarray}
 Then, coming back to our hessian computation, we have:
\begin{eqnarray}
 G^{A}(Hess J_2(\Omega)[V],W)&=&\int_{\partial\Omega}\left(1+AK^{2}_{c}\right)\left[\frac{-2AK^{3}_{c}}{\left(1+AK^{2}_{c}\right)^{2}}\psi\alpha+\frac{\partial\psi}{\partial \nu}\left(\frac{1}{1+AK^{2}_{c}}\right)\alpha\right.\nonumber\\&+&\left.\frac{1}{1+AK^{2}_{c}}\psi\left(\frac{3AK^{3}_{c}+K_{c}}{1+AK_{c}^{2}}\right)\alpha\right]\beta d\sigma,\nonumber\\
 &=&\int_{\partial\Omega}\left[\frac{-2AK^{3}_{c}}{1+AK^{2}_{c}}\psi\alpha+\frac{\partial\psi}{\partial \nu}\alpha+\psi\left(\frac{3AK^{3}_{c}+K_{c}}{1+AK_{c}^{2}}\right)\alpha\right]\beta d\sigma,\nonumber\\
 &=&\int_{\partial\Omega}\left[\frac{\partial\psi}{\partial \nu}+\psi\left(\frac{AK^{3}_{c}+K_{c}}{1+AK_{c}^{2}}\right)\right]\alpha\beta d\sigma,\nonumber\\
&=&\int_{\partial\Omega}\left[\frac{\partial\psi}{\partial \nu}+\psi K_{c}\left(\frac{1+AK^{2}_{c}}{1+AK_{c}^{2}}\right)\right]\alpha\beta d\sigma.\nonumber
\end{eqnarray}
Replacing $\psi$ by its expression, we have:
\begin{eqnarray}\label{punk2}
G^{A}(Hess J_2(\Omega)[V],W)=
\int_{\partial\Omega}\left[\frac{\partial}{\partial \nu}\left( Hu^2 + \gamma \right)+ K_{c}\left(Hu^2 + \gamma \right)\right]\big<V,\nu\big>\big<W,\nu \big> d\sigma.
\end{eqnarray}
\end{proof}
\begin{remark}
  Let us note first  that there is a symmetry relation with respect to the hessian which is in the case of our considered Riemannian structure a self adjoint operator with respect to the metric $G^A.$
\end{remark}
Let us have a look at the two formulas of the second derivation when $V= W= \alpha \nu.$\\
On the other hand by Theorem $\ref{punk0},$ we have:
\begin{eqnarray*}
G^{A}\left(Hess J(\Omega)[V], W \right)&=&d\left(dJ(\Omega)[W]\right)[V]-dJ(\Omega)[\nabla_{V}W]\nonumber.
\end{eqnarray*}
Then for $V=W$ we derive:
\begin{eqnarray*}
d\left(dJ(\Omega)[V]\right)[V]= d^{2}J (\O;V;V)=  G^{A}\left(Hess J(\Omega)[V], V \right)+dJ(\Omega)[\nabla_{V}V]\nonumber.
\end{eqnarray*}
From these information we can deduce the following conclusions as a corollary.
\begin{corollary}
\begin{itemize}
\item What is obtained with the Riemannian hessian formula is easier to derive simple control for the characterization of the optimal shape in a number of ways.
\item If the shape optimization problem  introduced in the subsection 6.3  has a minimum constrained with the eigenvalue  Laplacian-Neumann problem, then  $G^{A}\left(Hess J_2(\Omega)[V], V \right)\geq 0.$ 
The optimality condition is given by
\begin{eqnarray*}
		dJ_2 (\Omega)[V]:= dJ_2 (\Omega, V)= \int_{\partial \Omega} 2 u u' + (V. \nu) [\frac{\partial u^2}{\partial \nu} + H u^2]+ \gamma  V. \nu= 0
		\end{eqnarray*}
		And one interesting way to have $\Omega$ equal to disc is to look  for it, with $u= c_1$ on $\partial \Omega, c_1\in  \mathbb{R}^*.$ And if  the answer is positive then
		$$H u^2 + \gamma = 0 \, \, \mbox{on}\, \, \partial \Omega.$$
		
		And the inequality $G^{A}\left(Hess J_2(\Omega)[V], V \right)\geq 0$  is equivalent to $$\displaystyle \int_{\partial\Omega}\left[\frac{\partial}{\partial \nu}\left( Hu^2 + \gamma \right)\right]\alpha^2 d\sigma \geq 0, \, \, \forall \alpha \in \mathcal C^{\infty} (\mathbb{R}^2,  \mathbb{R}) .$$
		This  above integral is never but the following non negative quantity  $$c_1^2 \int_{\partial \Omega} H^2 \alpha^2 d\sigma.$$
\item Now, if $\Omega$ is only a critical point of the functional $J_2$, satisfying the   following symmetry problem
\begin{eqnarray*}
\begin{cases}\Delta u + k^2 u= 0\; \;  in \; \;  \Omega\\
\frac{\partial u}{\partial \nu}_{\big\arrowvert_{\partial \Omega}}= 0\\
 u_{\big \arrowvert_{\partial \Omega}}= const \neq 0\\
 k^2  > 0 , const \in \mathbb{R}
\end{cases}
\end{eqnarray*}
then we have
$$H (const)^2 + \gamma = 0 \, \, \mbox{on}\, \, \partial \Omega.$$

 And in addition $\Omega$ could be  a  good candidate of  strict local  minimum for $J_2$. In fact:
 \begin{eqnarray*}
\displaystyle \int_{\partial\Omega}\left[\frac{\partial}{\partial \nu}\left(Hu^2 + \gamma \right)\right]\alpha^2 d\sigma=  (const)^2\int_{\partial \Omega} H^2 \alpha^2 d\sigma \geq C_0 \|\alpha\|_{L^2 (\partial \Omega)}^2, C_0= (const)^2 H^2(\sigma_0) >0, \sigma_0\in  \partial \Omega.
\end{eqnarray*}
\end{itemize}
We think that the numerical part  of these  orientations are  interesting to be addressed. And we would like to invite  the reader  to see the paper \cite{Schu} (Theorem $2.4$ and section 3).
\end{corollary}
\begin{remark}
Let us  introduce the following shape functional
$$J_3 (\Omega):= \int_{\Omega} |\nabla u|^2 dx + \gamma vol (\Omega), \gamma \in \mathbb{R_+^*}$$
 with $\Omega \in \mathcal O$ and the following eigenvalue problem
 \begin{eqnarray*}
\begin{cases} 
\Delta u + k^2 u= 0 \; \; in\; \;   \Omega\\
u\big \arrowvert_{\partial \Omega}= 0 \\
\displaystyle \int_{\Omega} u^2 dx = 1

\end{cases}
\end{eqnarray*}
By the same techniques as previously, we have :
 \begin{eqnarray*}
 grad J_3= \frac{1}{1+AH^2} (-\frac{1}{2}(\frac{\partial u}{\partial \nu})^2 + \gamma).\\
G^{A}(Hess J_3(\Omega)[V],W)=
\int_{\partial\Omega}\left[\frac{\partial}{\partial \nu}\left (-\frac{1}{2}(\frac{\partial u}{\partial \nu})^2 + \gamma)\right)+ K_{c}\left((-\frac{1}{2}(\frac{\partial u}{\partial \nu})^2 + \gamma) \right)\right]\big<V,\nu \big>\big<W, \nu\big> d\sigma.
\end{eqnarray*}
Thanks to the above information, a same analysis can be tried  on the Schiffer's problem  related to conjecture 5. But theoretically, there is not  a qualitative information claiming straightly that $\Omega$ is a disc. A numerical study could give additional information on the shape of the domaine $\Omega$ solution to the overdetermined problem
\begin{eqnarray*}
\begin{cases}\Delta u + k^2 u= 0\; \;  in \; \;  \Omega\\
\frac{\partial u}{\partial \nu}_{\big\arrowvert_{\partial \Omega}}= c \neq 0\\
 u_{\big \arrowvert_{\partial \Omega}}= 0\\
 k^2 =  > 0 , c \in \mathbb{R}
\end{cases}
\end{eqnarray*}
We  think also  that following  the Theorem $2.4$ and the section 3 in \cite{Schu}, numerical tests could be realized.
\end{remark}
\subsection{Necessary  condition of minimality for  the two models}
The  above  shape optimization problem $ J_3 (\Omega)$ with the eigenvalue Laplacian-Dirichlet problem is well understood, see for instance  \cite{BucBut}, \cite{HP} and \cite{Hen} even for additional details (with more general class of admissible  domains).\\
For the one with a Neumann condition is more delicate  and largely open. We restrict ourselves to situations for which we are confident of the existence of an extension operator.
Let us consider $B$ a ball of $\mathbb{R}^N,$ $\Omega \subset B$ and the following class $\mathcal S_k,  k \in (0, \infty)$ of open sets defined by:
\begin{itemize}
\item for all $\Omega \in \mathcal S_k,$ there exists a linear continuous extension operator $P_{\Omega}$ of $H^1 (\Omega)$ into $H^1 (B)$ with 
\item $\|P_{\Omega}\|\leq k.$
\end{itemize}
\begin{remark}
\begin{itemize}
\item Let $\Omega$ be a bounded Lipschitz domain. The injection $H^1 (\Omega) \hookrightarrow L^2 (\Omega)$ is then compact, and the spectrum of the Neumann- Laplacian consists only on eigenvalues:
\begin{eqnarray*}
0= \lambda_1 \leq \lambda_2 \leq \lambda_3\leq \cdots \leq \lambda_k\leq \cdots \rightarrow \infty.
\end{eqnarray*}
\item It  is well known that  if $\Omega_n \in \mathcal O, n\in \mathbb{N},$ then  the sequence of eigenvalues  $\lambda_{\Omega_n}$ converges to $\lambda_{\Omega}$ (cf \cite{BucBut}, Corollary 7.4.2; \cite{Hen}, Theorem 2.3.25. )
\item We have also   the well known results for $\Omega, \Omega_n \in \mathcal O$, (see for instance  \cite{HP}, Theorem 2.4.10, pp 59) on existence of subsequence $\Omega_{n_k}$ that converges to $\Omega$ in the sense
of Hausdorff, in the sense of characteristic functions and in the sense of compacts.
Moreover, $\overline{\Omega}_{n_k}$ and $\partial \overline{\Omega}_{n_k} $  converge in the sense of Hausdorff respectively to $ \overline{\Omega}$  and  $\partial \overline{\Omega} .$
\item And finally, we have the shape minimization problems with respectivement $J_2 (\Omega),$ on the admissible domains space $\mathcal O$ constrained to eigenvalue Laplace Neumann problem gets a solution.
\end{itemize}
\end{remark}

\end{document}